\documentclass{amsart}
\usepackage{amsmath, amsthm, amssymb}
\usepackage[all]{xy}
\usepackage{imakeidx}
\usepackage{multicol}
\newtheorem{Teo}{Theorem}[section]
\newtheorem{Prop}[Teo]{Proposition}
\newtheorem{Lema}[Teo]{Lemma}

\theoremstyle{definition}
\newtheorem{Def}[Teo]{Definition}

\newtheorem{Obs}[Teo]{Remark}

\newcommand{\N}{\mathbb{N}}
\newcommand{\C}{\mathbb{C}}

\newcommand{\Llr}{\Longleftrightarrow}
\newcommand{\lra}{\longrightarrow}

\newcommand{\VR}{\mathcal{O}}

\begin{document}
\title{The Scott topology of a rooted non-metric tree}

\date{}

\author{Josnei Novacoski}
\address{CAPES Foundation  \newline \indent Ministry of Education of Brazil \newline \indent Bras\'ilia/DF 70040-020 \newline \indent Brazil}
\email{jan328@mail.usask.ca}

\thanks{During the realization of this project the author was supported by a research grant from the program ``Ci\^encia sem Fronteiras" from the Brazilian government.}

\begin{abstract} 
In this paper we prove that the Scott topology $\mathfrak S$ on a rooted non-metric tree $\mathcal T$ is strictly coarser than the weak tree topology. Moreover, for each $t\in \mathcal T$, we consider a natural order $\preceq_t$ on $\mathcal T$ under which $t$ is the root of $\mathcal T$. Then the weak tree topology is generated by the union of the Scott topologies $\mathfrak S_t$ associated to $\preceq_t$.
\end{abstract}

\maketitle

\section{Introduction}
During my presentation at ``ALANT 3 -- Joint Conferences on Algebra, Logic and Number Theory" in B\c{e}dlewo, Poland, many questions were raised about the different topologies on a rooted non-metric tree (see Definition \ref{ROTNONTR}). For instance, is the weak tree topology (see Definition \ref{WTTOP}) the same as the Scott topology (see Definition \ref{SCTOP})? This paper serves to answer those questions.

Favre and Jonsson introduced the valuative tree in \cite{Fav_1}. They considered valuations centered at the ring $\C[[x,y]]$ of the formal Laurent series in two variables over the field of complex numbers. In order to axiomatize some properties of this object, they introduced the concept of a rooted non-metric tree. In \cite{Gra}, Granja studied the equivalent case for valuations centered at a fixed two-dimensional regular local ring. In both works, the definition of rooted non-metric tree is not satisfactory (see discussion about that in \cite{JNVT}). In \cite{JNVT}, we complete this definition and compare some natural topologies on a rooted non-metric tree.

Since a rooted non-metric tree $(\mathcal T,\preceq)$ is, by definition, a partially ordered set we can consider the Scott topology on it. In this paper, we prove the following:

\begin{Teo}\label{COMPTOP}
The Scott topology on $\mathcal T$ is strictly coarser than the weak tree topology.
\end{Teo}

For each point $t\in\mathcal T$, we can define an order $\preceq_t$ on $\mathcal T$, such that $(\mathcal T,\preceq)$ and $(\mathcal T,\preceq_t)$ have the same segments (see Definition \ref{WTTOP}) and under which $t$ is a root of $\mathcal T$. Theorem \ref{COMPTOP} is a consequence of the following stronger result:

\begin{Teo}\label{mainthm}
For each $t\in\mathcal T$, consider the Scott topology $\mathfrak S_t$ on $\mathcal T$ associated to $\preceq_t$. Then the weak tree topology on $\mathcal T$ is the topology generated by $\displaystyle\bigcup_{t\in \mathcal T}\mathfrak S_t$.
\end{Teo}

\section{The valuative tree}

\begin{Def}\label{ROTNONTR}
A \index{Rooted non-metric tree}\textbf{rooted non-metric tree} is a partially ordered set $(\mathcal{T},\preceq)$ such that:
\begin{description}
\item[(T1)] There exists a (unique) smallest element $t_0\in\mathcal{T}$.
\item[(T2)] Every set of the form $I_t=\{a\in\mathcal{T}\mid a\preceq t\}$ is isomorphic (as ordered sets) to a real interval.
\item[(T3)] Every totally ordered convex subset of $\mathcal{T}$ is isomorphic to a real interval.
\item[(T4)] Every non-empty subset $\mathcal{S}$ of $\mathcal{T}$ admits an infimum in $\mathcal{T}$.
\end{description}
\end{Def}

In \cite{JNVT}, we prove the following:

\begin{Lema}
Under conditions \mbox{\rm \textbf{(T1)}} and \mbox{\rm \textbf{(T2)}}, the following conditions are equivalent:
\begin{description}
\item[(T4)] Every non-empty subset $\mathcal S\subseteq \mathcal{T}$ admits an infimum.
\item[(T4')] Given two elements $a,b\in\mathcal{T}$, the set $\{a,b\}$ admits an infimum $a\wedge b$.
\end{description}
\end{Lema}

\begin{Obs}
The lemma above shows that if a partially ordered set for which conditions \mbox{\rm \textbf{(T1)}} and \mbox{\rm \textbf{(T2)}} hold is directed (with respect to reverse set inclusion), then its order is a directed complete partial order (with respect to reverse set inclusion).
\end{Obs}

We will now define some properties associated to a rooted non-metric tree.
\begin{Def}\label{WTTOP}
\begin{description}

\item[(i)] Given a non-empty subset $\mathcal{S}\subseteq\mathcal{T}$ we define the \index{Join}\textbf{join} $\displaystyle\bigwedge_{s\in\mathcal{S}}s$ of $\mathcal{S}$ to be the infimum of $\mathcal{S}$.

\item[(ii)] Given two elements $a,b\in\mathcal{T}$ we define the \index{Closed segment}\textbf{closed segment} connecting them by
\[
\left[a,b\right]:=\{c\in\mathcal{T}\mid (a\wedge b\preceq c\preceq a)\vee (a\wedge b\preceq c\preceq b)\}.
\]
We define $\left]a,b\right]$ and $\left[a,b\right[$ similarly.

\item[(iii)] For $t\in\mathcal{T}$ we define an equivalence relation on $\mathcal{T}\backslash \{t\}$ by setting
\[
a\sim_{t}b\Llr t\notin [a,b].
\]
For an element $a\in \mathcal{T}\backslash \{t\}$ its equivalence class will be denoted by $[a]_{t}$, i.e., $[a]_t=\{b\in \mathcal T\mid a\sim_t b\}$. Denote by $\mathcal T_t$ to the set $\{[a]_t\mid a\in\mathcal T\}$ of all equivalence classes under $\sim_t$

\item[(iv)] The \index{Topology!weak tree}\textbf{weak tree topology} on $\mathcal{T}$ is the topology generated by all the sets of the form $[a]_{t}$ for $a$ and $t$ running through $\mathcal T$.

\item[(v)] A \index{Parametrization}\textbf{parametrization} of a rooted non-metric tree is an increasing (or decreasing) mapping $\Psi:\mathcal{T}\lra\left[-\infty,\infty\right]$ such that its restriction to every totally ordered convex subset of $\mathcal{T}$ is an isomorphism (of ordered sets) onto a real interval.

\item[(vi)] Given an increasing parametrization $\Psi:\mathcal{T}\lra\left[1,\infty\right]$ we define a metric on $\mathcal{T}$ by setting
\[
d_{\Psi}(a,b)=\displaystyle\left(\frac{1}{\Psi(a\wedge b)}-\frac{1}{\Psi(a)}\right)+\left(\frac{1}{\Psi(a\wedge b)}-\frac{1}{\Psi(b)}\right).
\]
\end{description}
\end{Def}

In \cite{JNVT}, we prove the following two results:

\begin{Teo}\label{Comp}
Let $(\mathcal T,\preceq)$ be a rooted non-metric tree and let $\Psi:\mathcal T\lra [1,\infty]$ be a parametrization of $\mathcal T$. Then the weak tree topology on $\mathcal{T}$ is coarser than or equal to the topology associated with the metric $d_\Psi$.
\end{Teo}

\begin{Teo}\label{Thm_1}
If there is an element $t\in\mathcal{T}$ with uncountably many branches (i.e., if $|\mathcal{T}_{t}|>|\N|$), then the weak tree topology is not first countable. In particular, the metric topology given by any parametrization is strictly coarser than the weak tree topology.
\end{Teo}

We fix an element $t\in \mathcal T$ and define a relation $\preceq_t$ on $\mathcal T$ as follows. For any two elements $a,b\in \mathcal T$ we declare that $a\preceq_t b$ if and only if $a\in [t,b]$. It is straightforward to prove that $\preceq_t$ is an order on $\mathcal T$ and that $(\mathcal T,\preceq_t)$ is a rooted non-metric tree (for which $t$ is a root). Moreover, the segments under this new order are exactly the same as those defined by the order $\preceq$. Consequently, the weak tree topology on $\mathcal T$ defined by these two orders is the same.

The following lemma will be used in the proof of the main theorem.
\begin{Lema}\label{lemaoninterv}
Take elements $a,b,c$ in a rooted non-metric tree $\mathcal T$. Then we have $[a,c]\subseteq [a,b]\cup[b,c]$.
\end{Lema}
\begin{proof}
We have to prove that both segments $[a\wedge c,a]$ and $[a\wedge c, c]$ are subsets of $[a,b]\cup [b,c]$. 

Consider the segment $I_b$, which is totally ordered by property \textbf{(T2)}. Since $b\wedge c\preceq b$ and $a\wedge b\preceq b$, these two elements are comparable. If $a\wedge b\preceq b\wedge c$, then $a\wedge b\preceq c$. Hence, $a\wedge b\preceq a\wedge c$, which implies that $[a\wedge c, a]\subseteq [a\wedge b, a]\subseteq [a,b]$.

If $b\wedge c\preceq a\wedge b$, then $b\wedge c\preceq a$ and consequently $b\wedge c\preceq a\wedge c$. On the other hand, since $a\wedge b\preceq a$ and $a\wedge c\preceq a$ we have that $a\wedge b\preceq a\wedge c$ or $a\wedge c\preceq a\wedge b$. In the first case we reason like in the previous paragraph to obtain $[a\wedge c, a]\subseteq [a,b]$. If $a\wedge c\preceq a\wedge b$, then $a\wedge c\preceq b$ and consequently $a\wedge c\preceq b\wedge c$. Hence, $a\wedge c=b\wedge c$ and thus
\[
[a\wedge c,a\wedge b]=[b\wedge c,a\wedge b]\subseteq [b\wedge c,b]\subseteq [b,c].
\]
Therefore
\[
[a\wedge c, a]= [a\wedge c,a\wedge b]\cup [a\wedge b,a]\subseteq [b, c]\cup [a,b].
\]

The proof that $[a\wedge c,c]\subseteq [a,b]\cup [b,c]$ is analogous.
\end{proof}

\section{The Scott topology}

Consider a partially ordered set $(\mathcal P,\preceq)$. A subset $\mathcal S$ of $\mathcal P$ is said to be an \textbf{upper set} if for every elements $x,y\in \mathcal P$, if $x\in \mathcal S$ and $y\geq x$, then $y\in \mathcal S$. The set $\mathcal S$ is said to be \textbf{inaccessible by directed joints} if for every directed subset $\mathcal D$ of $\mathcal P$, if $\sup \mathcal D\in \mathcal S$, then $\mathcal D\cap \mathcal S\neq\emptyset$.
\begin{Def}\label{SCTOP}
The Scott topology $\mathfrak S$ on $\mathcal P$ is defined by setting as open sets the upper sets which are inaccessible by directed joints.
\end{Def}

\begin{Prop}\label{Proposcwtt}
Let $(\mathcal T, \preceq)$ be a rooted non-metric tree. Then every Scott open set $\VR$ of $\mathcal T$ is open in the weak tree toplogy.
\end{Prop}
\begin{proof}
For each point $a\in \VR$, we will prove that there exists $t\in\mathcal T$ such that $[a]_{t}\subseteq \VR$. Consider the set $\mathcal D:=\{t'\in \mathcal T\mid t'\preceq a\mbox{ and } t'\neq a\}$. By property \textbf{(T2)}, this set is order isomorphic to a real interval (thus a directed set). Hence, $\sup \mathcal D=a\in \VR$ and since $\VR$ is open in the Scott topology, we have that $\VR\cap \mathcal D\neq\emptyset$.

Take any $t\in \VR\cap \mathcal D$. For each $b\in [a]_{t}$, we will prove that $b\in \VR$. Suppose, towards a contradiction, that $t\not\preceq b$. Since $t\preceq a$ and $I_a$ is totally ordered we have $t\preceq a\wedge b$ or $a\wedge b \preceq t$. The first case cannot happen because $t\not\preceq b$ and $a\wedge b\preceq b$. Consequently, $a\wedge b \preceq t\preceq a$ and hence $t\in[a,b]$. This means that $b\notin [a]_t$, a contradiction. Hence, $t\preceq b$. Since $t\in \VR$ and $\VR$ is an upper set we obtain that $b\in \VR$. Therefore, $[a]_t\subseteq \VR$, which concludes our proof.
\end{proof}

\begin{proof}[Proof of Theorem \ref{COMPTOP}]
The previous proposition shows that the Scott topology is coarser than the weak tree topology. It remains to show that they are different.

Take $t\in\mathcal T$ such that $t$ is not the root of $(\mathcal T,\preceq)$. Consider the open subbasic set $[t_0]_t$ in the weak tree topology. Then $t_0\in [t_0]_t$, $t_0\preceq t$ but $t\notin [t_0]_t$ which implies that $[t_0]_t$ is not an upper set. Hence, $[t_0]_t$ is not open in the Scott topology.
\end{proof}

\begin{Obs}
There are many ways to see that the Scott topology is not the weak tree topology. For instance, the weak tree topology is ways Hausdorff, but the Scott topology (of a rooted non-metric tree) is not. Also, if we consider the orders $\preceq_t$ and $\preceq_s$ on $\mathcal T$ for $t\neq s$, then the associated weak tree topologies are the same, but the Scott topologies $\mathfrak S_t$ and $\mathfrak S_s$ are not.
\end{Obs}

The next result will be useful to prove Theorem \ref{mainthm}.
\begin{Lema}\label{inaccdirjoin}
For the rooted non-metric tree $(\mathcal T,\preceq)$, every subasic open set in the weak tree topology is inaccessible by directed joints.
\end{Lema}
\begin{proof}
Consider a subbasic open set $[a]_t$ of $\mathcal T$ and take a directed set $\mathcal D$ such that $\mathcal D\subseteq \mathcal T\setminus [a]_t$. We want to prove that $\sup \mathcal D=:b\notin [a]_t$. If $\mathcal D=\{t\}$, then $\sup \mathcal D=t\notin [a]_t$ and we are done. Hence, assume that there exists $d\in \mathcal D$ such that $d\neq t$.

Assume first that $t\not \preceq a$. We will show that $t\preceq \mathcal D$ (i.e., $t\preceq d$ for every $d\in\mathcal D$). Consequently, also $t\preceq b$ and thus $b\notin [a]_t$. For an element $d\in \mathcal D$, if $t\preceq a\wedge d$, then in particular $t\preceq a$, which is a contradiction. Hence, $t\not\preceq a\wedge d$. If $t\not\preceq d$, then $t\notin [d,a]$ which is a contradiction to $\mathcal D\subseteq \mathcal T\setminus [a]_t$. Hence, $t\preceq d$, which is what we wanted to prove.

Assume now that $t\preceq a$. If we prove that $a\wedge b\preceq t$, then $t\in [a,b]$ and this will conclude our proof. Since $t\preceq a$ and $a\wedge b\preceq a$, we have that $t\prec a\wedge b$ of $a\wedge b\preceq t$. Suppose, towards a contradiction, that $t\prec a\wedge b$. If $a\wedge b\prec b$, then there exists $d\in \mathcal D$ such that $a\wedge b \prec d\preceq b$. This implies that $a\wedge b=a\wedge d$ and consequently $t\notin [a,d]$. This is a contradiction to $d\notin [a]_t$. If $a\wedge b=b$, then $b\preceq a$. Since $b=\sup \mathcal D$, there exists $d\in D$ such that $t\prec d\preceq b\preceq a$ and consequently $t\notin [a,d]$, which is again a contradiction.
\end{proof}

\begin{Lema}\label{upperset}
If $t$ is the root of $\mathcal T$, then the set $[a]_t$ is an upper set for every $a\in \mathcal T\setminus \{t\}$.
\end{Lema}
\begin{proof}
Take elements $b,c\in \mathcal T$ such that $b\in [a]_t$ and $b\prec c$. We will show that $c\in [a]_t$. Since $b\prec c$, we have $[b,c]=\{d\in \mathcal T\mid b\preceq d\preceq c\}$. Since $t$ is the smallest element of $\mathcal T$ we have that $t\notin [b,c]$. On the other hand, since $b\in [a]_t$, by the definition of $[a]_t$, we also have that $t\notin [a,b]$. Therefore, we use Lemma \ref{lemaoninterv} to conclude that $t\notin [a,c]\subseteq[a,b]\cup[b,c]$. Therefore, $c\in [a]_t$.
\end{proof}

We are now ready to prove our main result.

\begin{proof}[Proof of Theorem \ref{mainthm}]
As a consequence of Lemma \ref{upperset}, every set of the form $[a]_t$ is an upper set with respect to the order $\preceq_t$. Moreover, applying Lemma \ref{inaccdirjoin} to $(\mathcal T,\preceq_t)$ we obtain that this set is inaccessible by directed joints. Hence, $[a]_t\in \mathfrak S_t$.

On the other hand, by Proposition \ref{Proposcwtt}, every open set in the Scott topology $\mathcal S_t$ is open in the weak tree topology. Therefore, the weak tree topology is generated by $\displaystyle\bigcup_{t\in\mathcal T}\mathfrak S_t$.
\end{proof}

\end{document}